\documentclass{article}
\usepackage[utf8]{inputenc}
\usepackage{amssymb}
\usepackage{amsmath}
\usepackage{amsthm}
\usepackage{thmtools}
\usepackage{tabularx}
\usepackage{listings}
\usepackage{stmaryrd}
\usepackage{graphicx}
\usepackage{cmap}
\usepackage{amssymb,amsmath,amsthm}
\usepackage{euscript}
\usepackage[english]{babel}
\usepackage[dvipsnames]{xcolor}

\newtheorem{theorem}{Theorem}[section]
\newtheorem*{theorem*}{Theorem}
\newtheorem{lemma}[theorem]{Lemma}

\newtheorem{conjecture}[theorem]{Conjecture}
\newtheorem{corollary}[theorem]{Corollary}
\newtheorem{prop}[theorem]{Proposition}
\theoremstyle{definition}

\newcommand{\fa}{\forall}
\newcommand{\lt}{\left}
\newcommand{\rt}{\right}
\newcommand{\N}{{\mathbb N}}

\newcommand{\R}{{\mathbb R}}

\newcommand{\Rarr}{\Rightarrow}

\newcommand{\Len}{\mathop{\rm Len\mskip 1mu}}

\newcommand{\im}{\mathop{\rm Im}\nolimits}

\newcommand{\Le}{\leqslant}
\newcommand{\Ge}{\geqslant}
\newcommand{\mto}{\mapsto}

\newcommand{\vphi}{\varphi}
\newcommand{\veps}{\varepsilon}

\newcommand{\lims}{\limits}
\newcommand{\Inf}{\infty}

\newcommand{\eq}{\equiv}

\title{On the Complexity of Horizontal Unitracks.}

\author{Ivan Molodyk}

\date{}

\begin{document}

\maketitle

\begin{abstract}
    This paper concerns the geometry of bicycle tracks. We model bicycle as an oriented segment of a fixed length that is moving in the Euclidean plane so that the trajectory of the rear point is tangent to the segment at all times. The trajectories of front and back points of the segment are called bicycle tracks, and one asks if it is possible that the front track is contained in the rear track (other than when they are straight lines). Such curves are called unitracks or unicycle tracks. In 2002 D.Finn proposed a construction of unitracks that are obtained as a union of a sequence of curves. Numerical evidence suggested that these curves behave expansively and that various numerical characteristics of the curves grow quickly in the sequence. In this paper we prove that the curves that form a unitrack in Finn's construction cannot remain graphs of functions, unless they are straight lines. We conclude that the horizontal amplitude of the curves has a linear growth rate between $1$ and $2$.
\end{abstract}

\large 

\section{Introduction}

The bicycle tracks geometry is a rich and fascinating subject, and it attracted much attention lately. 

One studies a simple model of a bicycle as an oriented segment of fixed length moving in the Euclidean plane. Let the front and the rear points of the segment represent the touching points of the front and back wheels of the bicycle with the ground. Unlike the front wheel, the rear wheel of the bike is fixed on the frame, so its trajectory is tangent to the moving segment at all times. This is a non-holonomic constraint on the motion that defines a completely non-integrable distribution on the 3-dimensional configuration space of directed segments of a fixed length.
The same model describes the motion of hatchet planimeter, see \cite{Fo,FLT}.

To put things into perspective, we briefly mention the main directions of study of this bicycle model. 

The non-holonomic ``bicycle" constraint can be expressed as a first order differential equation that, in appropriate coordinates, is a Riccati equation. The monodromy of this equation is a M\"obius transformation depending on the front track. If the front track is a closed curve, the monodromy map assigns the final position of the bicycle to its initial position. The bicycle path is closed if the monodromy has a fixed point. An old conjecture, due to an engineer Menzin, is that if the bicycle segment is unit, the front track is a closed convex curve that bounds area greater than $\pi$, then the monodromy is hyperbolic, that is, has two fixed points. See \cite{LT} for a proof. Another paper on this topic is \cite{BHT}. Also see \cite{HPZ} for a version in the spherical and hyperbolic geometries. 

Two front tracks that share the same closed rear track traversed in the opposite directions are said to be in the bicycle correspondence. This correspondence is completely integrable and it is closely related with the planar filament equation, a completely integrable partial differential equation describing evolution of plane curves; see \cite{BLPT,Ta}. 

A question that was explored in \cite{Fi2} is as follows: if the bicycle tracks were two closed curves on the plane -- would we be able to tell what  the direction of motion was? This question relates bicycle dynamics to a different geometrical problem concerning bodies floating in equilibrium. The connection and certain properties of such bodies were studied in  \cite{We}.

One can consider the bicycle motion as a horizontal curve in the configuration space, that is, a curve that is tangent to the distribution describing the non-holonomic constraint. This distribution has a natural metric, and one wants to describe its geodesics. In this metric, the length of the bicycle path is the length of the front track. See \cite{ABLMS,BJT,PT} for results in this direction. 

Still another question is whether one can ride a bicycle in such a way that the tracks of the rear and the front wheels coincide (other than along a straight line)? The solutions to this problem are called unicycle tracks, or unitracks. D. Finn have proposed a construction of a unitrack in \cite{Fi1}. We should also mention a different approach to constructing a unitrack due to S. Wagon \cite{Wa}. However, this paper concerns some of the properties of Finn's construction, briefly explained below.

Consider a smooth curve connecting $(0, 0)$ to $(1, 0)$ on the plane, such that it is infinitely flat with respect to the horizontal direction at the endpoints. We assume that this curve $\gamma_0$ is a segment of the rear track of a bicycle of length $1$. Then we can reconstruct the corresponding segment of the front track --- we call it $\gamma_1$, and it will connect points $(1, 0)$ and $(2, 0)$. Due to the specific conditions of "flatness" at the endpoints, $\gamma_0$ and $\gamma_1$ can be "glued" together smoothly at the point $(1, 0)$, therefore we can assume $\gamma_1$ is the continuation of the rear track that started as $\gamma_0$. In the same way we can proceed to construct curves $\gamma_2, \gamma_3$, and so on. The union of this sequence of curves forms a smooth rear track that contains the front track as a subset. \\

\begin{figure}[hb!]\label{pic:unitrack}
    \centering
    \includegraphics[height=6.5cm]{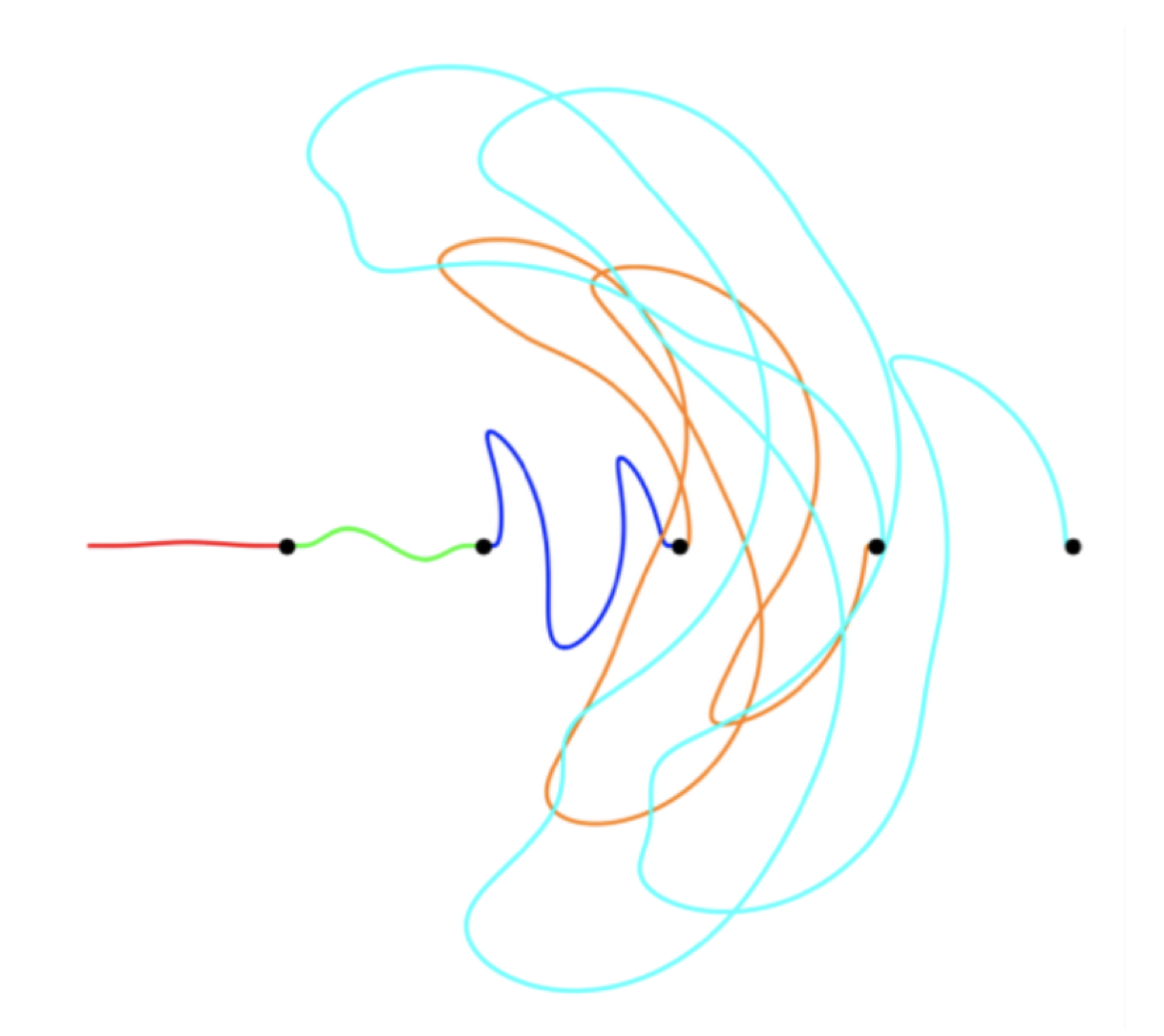}
    \caption{The first 5 segments of a unitrack with $\gamma_0$ being the graph of a function $f(t) = 4e^{\frac{-1}{t(1-t)}}$. Image credit to Stan Wagon.}
\end{figure}

As we can see from Figure \ref{pic:unitrack}, the curves in the sequence $\gamma_0, \gamma_1, \gamma_2, ...$ quickly become quite complicated. Their lengths grow, self-intersections appear, segments of great curvature appear, etc. Here are some of the theorems and conjectures regarding this observation in an informal formulation:

\begin{itemize}
    \item 
        \textbf{Theorem A}: The length of $\gamma_n$ increases with $n$.\\

    \item 
        \textbf{Theorem B}: In fact, the length of $\gamma_n$ tends to infinity with $n$.\\

    \item 
        \textbf{Theorem C}: The `oriented area' bounded between the curves and the horizontal axis remains the same as $n$ grows.\\
        
    \item 
        \textbf{Theorem D}: The number of `zeros' of the curves $\gamma_n$ strictly grows with $n$.\\

    \item 
        \textbf{Conjecture E}: The curves $\gamma_n$ develop self-intersections.\\

    \item
        \textbf{Theorem F}: The `vertical amplitude' of $\gamma_n$ grows with $n$.\\

    \item 
        \textbf{Conjecture G}: The `vertical amplitude' of $\gamma_n$ tends to infinity with $n$.\\

    \item 
        \textbf{Theorem H}: The `horizontal amplitude' of $\gamma_n$ tends to infinity with $n$.\\

    \item
        \textbf{Theorem I}: The curves $\gamma_n$ cannot remain graphs of functions as $n$ grows.
\end{itemize}

Theorems A, C, and D are reasonably simple to prove statements. Theorem F is not particularly complicated either, although, the proof is quite technical. On the other hand, Theorems B, H, and I are significantly more complicated. In Section \ref{proof1} of this paper we will provide a proof of Theorem I by contradiction. We will assume that in fact there is a curve $\gamma_0$ that will produce a sequence of curves $\gamma_n$, where all the curves will be graphs of functions, and we will trace the horizontal movement of points on curves within a bounded interval. Then we will analyze the limiting distribution of the points along the horizontal line, and obtain a bound on the length of the initial curve. When we conclude that the length of the initial curve was no greater than $1$, and therefore $\gamma_0$ must be a straight line, the proof of Theorem I will be completed. Theorem H follows from Theorem I (the proof is in Section \ref{proof2}), and Theorem B follows from Theorem H immediately.

\section{Basic Properties}

    Let $X$ be the set of all infinitely smooth curves, parametrized by closed intervals, and immersed in the Euclidean plane. We define the map:
    $$
    \vphi\colon X\to X;\quad \vphi\colon \gamma \mto \gamma + \frac{\dot{\gamma}}{\|\dot{\gamma}\|},
    $$
    where the addition is pointwise, i.e., $\vphi(\gamma)(t) = \gamma(t) + \frac{\dot{\gamma}(t)}{\|\dot{\gamma}(t)\|}$ as vectors in the Euclidean plane.\\

    \begin{figure}[hbt!]
        \centering
        \includegraphics[height=4cm]{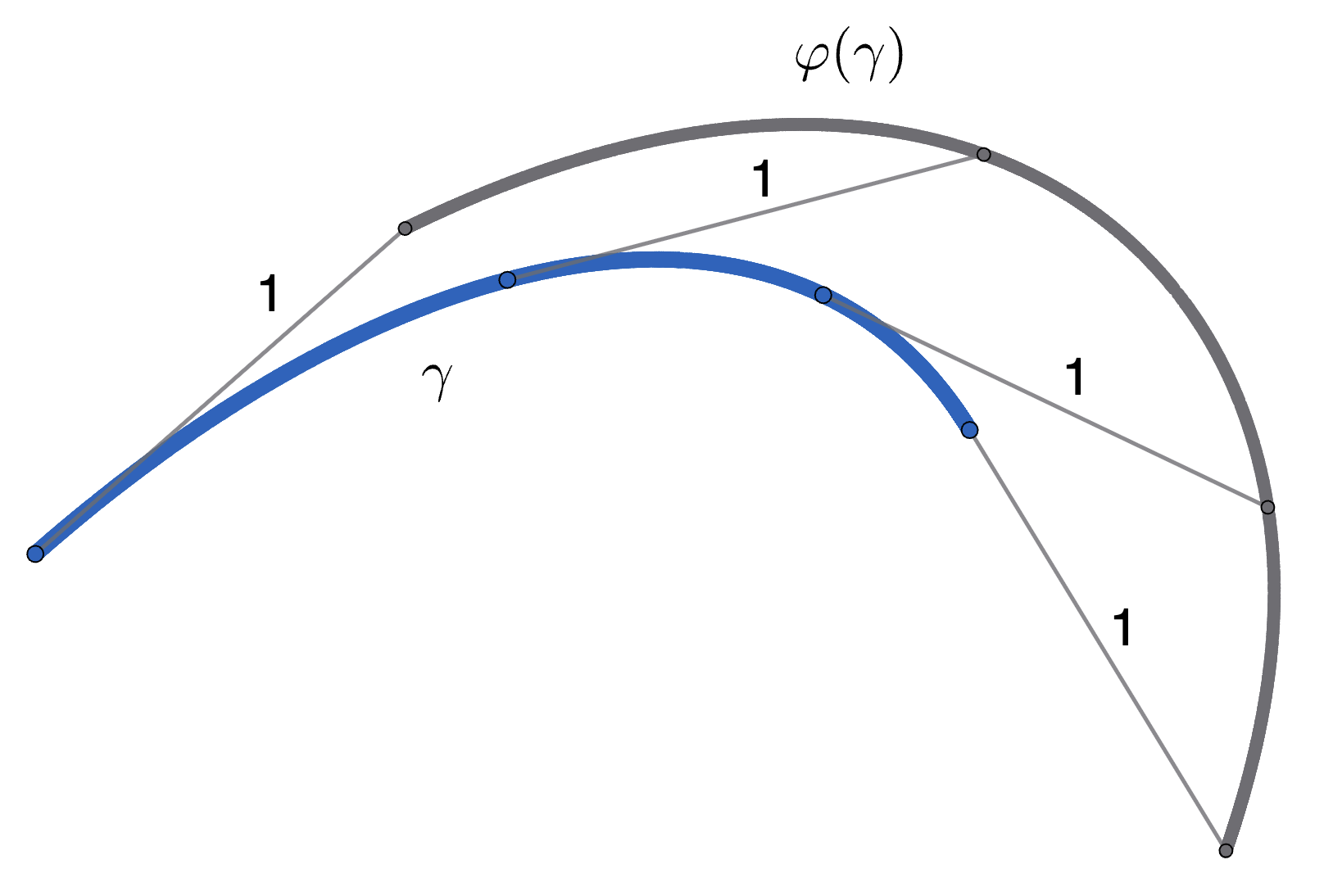}
        \caption{The curve $\vphi(\gamma)$ is the front track of a bicycle, where $\gamma$ is the back track.}
    \end{figure}

    \begin{lemma}
    \label{length}
        If $\gamma\in X$, then $\Len(\vphi(\gamma)) \Ge \Len(\gamma)$, where $\Len$ is length.
    \end{lemma}
    \begin{proof}[Proof]
        Let $s$ be a natural parameter on $\gamma$. Then:
        $$
        \vphi(\gamma) = \gamma + \dot{\gamma};\quad \dot{\vphi(\gamma)} = \dot{\gamma} + \ddot{\gamma}.
        $$
        Since $s$ is a natural parametrization of $\gamma$, we conclude that $\ddot{\gamma}\perp\dot{\gamma}$. Therefore:
        \begin{multline*}
        \Len(\vphi(\gamma)) = \int\lims_{0}^{\Len(\gamma)}\!\!\|\dot{\vphi(\gamma)}\|ds = \int\lims_{0}^{\Len(\gamma)}\!\!\!\|\dot{\gamma}+\ddot{\gamma}\|\,ds = \int\lims_{0}^{\Len(\gamma)}\!\!\sqrt{\|\dot{\gamma}\|^2 + \|\ddot{\gamma}\|^2}\,ds = \\ = \int\lims_{0}^{\Len(\gamma)}\sqrt{1 + k_\gamma^2(s)}\,ds \Ge \int\lims_{0}^{\Len(\gamma)}\!\! 1 \,ds = \Len(\gamma),
        \end{multline*}
        where $k_\gamma(s) = \|\ddot{\gamma}(s)\|$ is the modulus of the curvature of $\gamma$. Not only we proved that the length of the curve doesn't decrease under the transformation $\vphi$, but also that it strictly increases unless its curvature is constant zero, meaning $\gamma$ is a straight line.
    \end{proof}

    The following proposition features some `expansive' properties of the map $\vphi$:
    
    \begin{prop}
    \label{loc-maximum}
        Consider a smooth curve $\gamma = (x(t), y(t))\colon [a, b]\to\R^2$ and let $\vphi(\gamma) = (u(t), v(t))$. Assume that $\gamma$ is not a horizontal segment, meaning $y(t)\ne const$. Then assuming that the maximum value of $y$ is reached in the interior of $[a, b]$, the maximum value of $v$ is bigger than the maximum value of $y$.
    \end{prop}

    The proof of this proposition consists of a somewhat tedious page-long calculation. Since we do not use this result in the paper, we will not include the proof here.
    
\section{Horizontal Unitrack} \label{hor-uni}

    We are going to study a variation of the unicycle construction, proposed by D.Finn, and presented in the introduction. We will work with a smaller set of curves, defined below.\\

    Let $Y\subset X$ be the space of immersed curves $\gamma\colon[a, b]\to\R^2$ such that:
    \begin{itemize}
        \item $\gamma(a) = (0, 0)$ and $\gamma(b) = (1, 0)$;
        \item $\dot{\gamma}(a) = \dot{\gamma}(b) = (1, 0)$ in the natural parameter, or, $\dot{\gamma}(a)$ and $\dot{\gamma}(b)$ are horizontal vectors oriented to the right;
        \item $\fa n\Ge 2\,\, \gamma^{(n)}(a) = \gamma^{(n)}(b) = (0, 0)$.
    \end{itemize}
    In other words, $Y$ is the set of smooth curves immersed in $\R^2$, starting at the origin, finishing in $(1, 0)$, and infinitely flat with respect to the horizontal direction at the endpoints.\\

    \begin{figure}[hbt!]
        \centering
        \includegraphics[height=4.3cm]{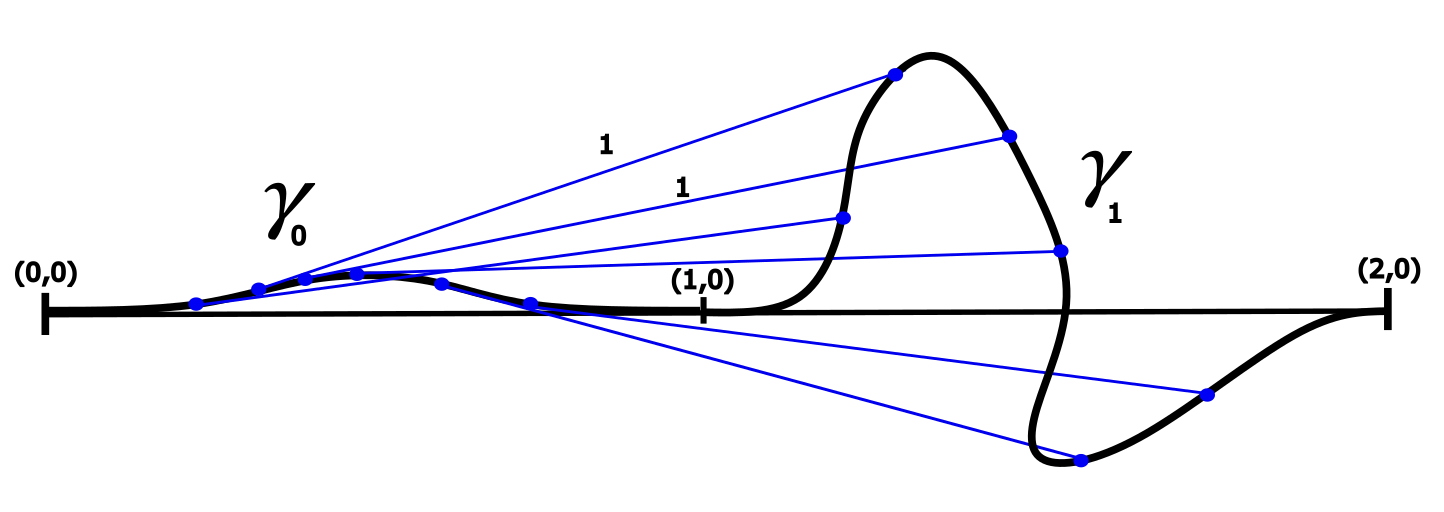}
        \caption{Curve $\gamma_0$ defines the next curve $\gamma_1$}
    \end{figure}
    
    Recall that in Finn's construction of a unicycle track we start with a curve $\gamma_0\in Y$, and produce the sequence of curves $\gamma_n = \vphi^n(\gamma_0)$. Due to the conditions defining $Y\subset X$, the union of the images of these curves forms an infinitely smooth curve. If a bicycle starts riding with its rear wheel along this curve, starting at the beginning of $\gamma_0$, the front wheel will be riding along the same curve, starting at the beginning of $\gamma_1$.\\

    For convenience we will introduce the modification of the initial mapping $\vphi$, but abusing notation, we will refer to it by the same letter $\vphi$ throughout the paper:
    
    \begin{align}
    \label{formula}
    \vphi\colon Y\to Y;\quad \vphi(\gamma) = \gamma + \frac{\dot{\gamma}}{\|\dot{\gamma}\|} - (1, 0).
    \end{align}

    \begin{prop}
        The new map $\vphi$ is well defined, i.e., for all $\gamma\in Y$ its image $\vphi(\gamma)$ also belongs to $Y$.
    \end{prop}
    
    \begin{proof}
        We need to check that the three conditions from the definition of $Y$ are satisfied by $\vphi(\gamma)$:
        \begin{itemize}
            \item $\vphi(\gamma)(a) = \gamma(a) + \frac{\dot{\gamma}(a)}{\|\dot{\gamma}(a)\|} - (1, 0) = (0, 0) + \frac{(1, 0)}{1} - (1, 0) = (0, 0)$,\\
            $\vphi(\gamma)(b) = \gamma(b) + \frac{\dot{\gamma}(b)}{\|\dot{\gamma}(b)\|} - (1, 0) = (1, 0) + \frac{(1, 0)}{1} - (1, 0) = (1, 0)$

            \item Consider the arc length parameter on a curve $\gamma$. Then
            $$
            \dot{\vphi(\gamma)}(b) = \dot{\vphi(\gamma)}(a) = \lt(\gamma(a) + \dot{\gamma}(a) - (1, 0)\rt)\dot{} = \dot{\gamma}(a) + \ddot{\gamma}(a) - (0, 0) = (1, 0) + (0, 0) = (1, 0)
            $$

            \item For all $n > 2$ the following holds in the arc length parameter:
            $$
            \vphi(\gamma)^{(n)} = \lt(\gamma(a) + \dot{\gamma}(a) - (1, 0)\rt)^{(n)} = \gamma^{(n)}(a) + \gamma^{(n+1)}(a) - (0, 0) = (0, 0).
            $$
            Therefore, the curve $\vphi(\gamma)$ is an element of $Y$.
        \end{itemize}
    \end{proof}

\section{Horizontal Unitrack Properties}\label{sec4}
    Denote the projections of $\gamma_n$ on the horizontal and vertical axes by $x_n$ and $y_n$ --- two functions from the parameter interval to $\R$. Then the following  holds:
    \begin{prop}
    \label{zeros}
        The number of `zeros' of $\gamma_n$ increases with $n$ at least by one. More precisely, the number of solutions of the equation $y_{n+1}=0$ is at least one bigger than the number of the solutions of $y_n=0$, if this number if finite. 
    \end{prop}
    The proof can be found in \cite{LT}.
    \begin{theorem}
    \label{graphs}
        Assume that for every $n\Ge 0$ the curve $\gamma_n$ with a parameter $t$ satisfies $y_n(t) = f_n(x_n(t))$ for some smooth function $f_n\colon[0, 1]\to\R$, i.e., $\gamma_n$ is a graph of a smooth function. Then for all $n$ the function $f_n(x)$ is a constant zero function. In other words, unless the unitrack is trivial, after some $n$ the curves $\gamma_n$ will fail to remain graphs of functions.
    \end{theorem}
    The next section of this paper will be devoted to the proof of this theorem. Before that, we introduce more definitions and present some conjectures.\\
    
    Define the \textit{horizontal amplitude} of the curve $\gamma(t) = (x(t), y(t))$ with a parameter $t\in[a, b]$ as:
    $$
    H(\gamma) = \max\lims_{t\in[a, b]}{x(t)} - \min\lims_{t\in[a, b]}{x(t)}.
    $$
    Similarly, the \textit{vertical amplitude} of a curve $\gamma$ is:
    $$
    V(\gamma) = \max\lims_{t\in[a, b]}{y(t)} - \min\lims_{t\in[a, b]}{y(t)}.
    $$
    \begin{theorem}
    \label{hor-ampl}
        Assuming that $\gamma_0$ is not trivial, the sequence $H(\gamma_n)$ is unbounded and non-decreasing. Moreover, there exist $c_1, c_2 > 0$ depending on $\gamma_0$, such that:
        $$
        n - c_1 \Le H(\gamma_n) \Le 2n - c_2
        $$
    \end{theorem}
    This theorem follows from Theorem \ref{graphs}, and the proof will be provided in the 6-th section of the paper.
    \begin{conjecture}
    \label{vert-ampl}
        The sequence $V(\gamma_n)$ is unbounded.
    \end{conjecture}

    One may notice that Theorem \ref{hor-ampl} and Conjecture \ref{vert-ampl} are very similar. However, there is a very significant difference between horizontal and vertical directions in the horizontal unitrack construction. In the modified version of $\vphi$ it is reflected in the last term of the formula (\ref{formula}).\\
    
    \textit{Remark:} the fact that $V(\gamma_n)$ is an increasing sequence, follows from the Proposition \ref{loc-maximum}, as both global maximum and global minimum of $y_{n+1}$ are greater and smaller than global maximum and minimum of $y_n$ respectively.\\

    \begin{conjecture}
        If $\gamma_0$ is not trivial, then there exists $n$, such that $\gamma_n$ has self-intersections. 
    \end{conjecture}

    There is also a weaker variation of this conjecture, that is still open.

    \begin{conjecture}
        If $\gamma_0$ is not trivial, then the unitrack consisting of the curves $\gamma_n$, where $n\in\N$, has self-intersections. 
    \end{conjecture}

\section{Proof of Theorem \ref{graphs}}
\label{proof1}

    \subsection{Plan of the proof}
    \begin{itemize}
        \item We assume the opposite --- that there exists a non-trivial curve $\gamma_0(t)$, such that all the curves $\gamma_n(t)$ are graphs of smooth functions. In cartesian coordinates if we denote the horizontal and vertical components of a curve $\gamma_n(t)$ by $x_n(t)$ and $y_n(t)$ we assume that $y_n(t) = f_n(x_n(t))$. We will also introduce a parameter $t\in[0, 1]$ on the curves, such that is coincides with the horizontal coordinate on the first curve, and is pushed forward on the other curves by $\vphi$:
        $$
        \gamma_0(t) = (t, f_0(t));\quad \gamma_{n+1}(t) = \gamma_n(t) + \frac{\dot{\gamma}_n(t)}{\|\dot{\gamma}_n(t)\|} - (1, 0)\quad \text{for all} \,\,n\Ge 0.
        $$

        \item We introduce the function $s_n$ that traces the action of $\vphi$ on horizontal coordinates of points on the curves (see Figure \ref{pics_n}). This function $s_n$ relates the horizontal components of the curves in the following way:
        $$
        x_{n+1}(t) = x_n(t) - s_n(x_n(t)).
        $$

        \begin{figure}[hbt!]
            \centering
            \includegraphics[height=5cm]{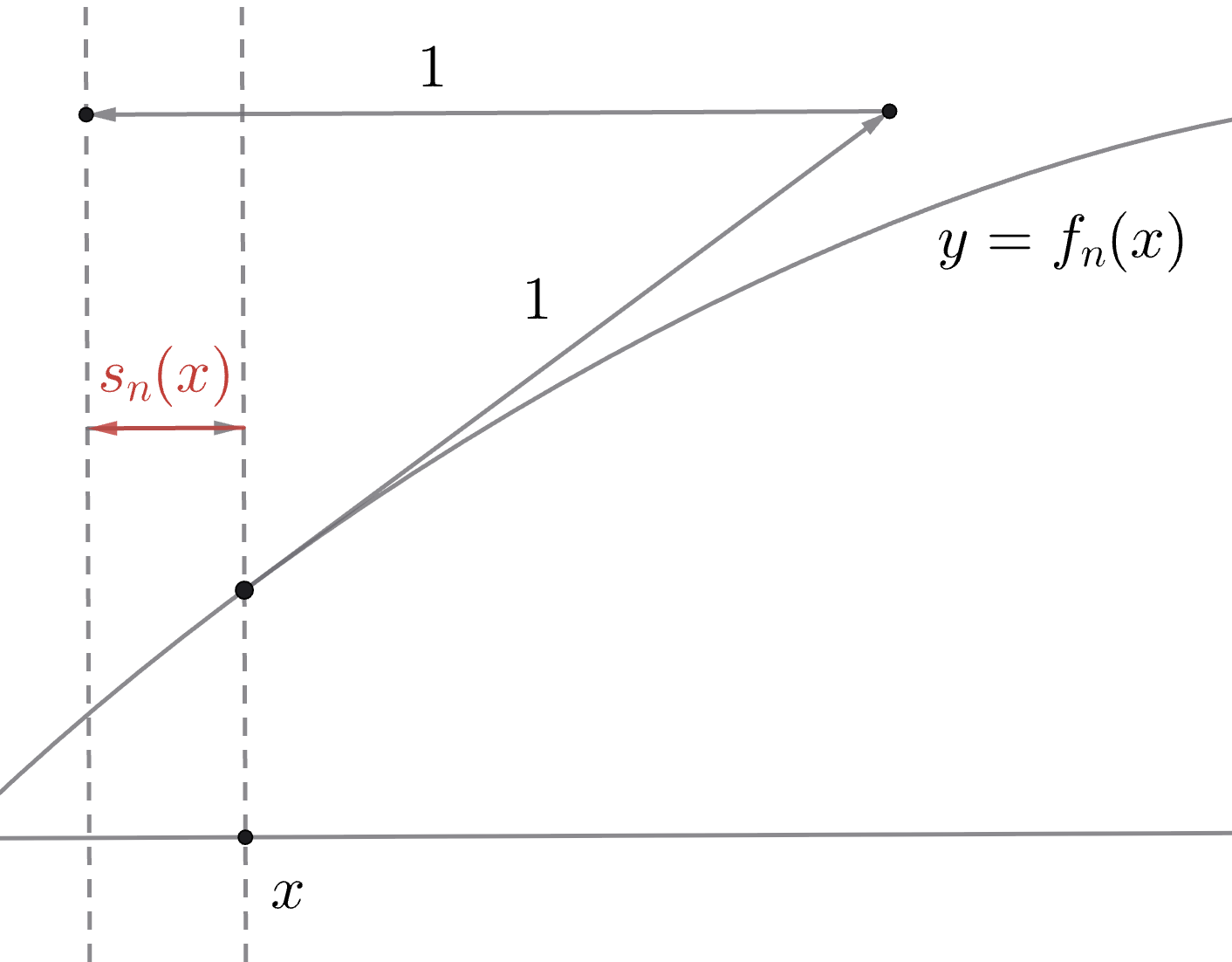}
            \caption{The definition of $s_n(x)$.}\label{pics_n}
        \end{figure}

        \item From the initial assumptions we analyze the sequence of functions $\{x_n\}$, and obtain a pointwise limit $L(t) = \lim\lims_{n\to\Inf}x_n(t)$, that inherits some, but not all of the properties of $x_n$. For example, $L$ is not necessarily continuous apriori, which causes some complications.\\

        \item Take a closed segment $[a, b]\subset[0, 1]$ and consider a very large $n\in\N$. The segment of the curve $\gamma_n$, defined by the parameter interval $[a, b]$, is "hanging above" a closed interval very close to $[L(a), L(b)]$. Denote its length by $\delta = L(b) - L(a)$. We obtain an upper bound on the slope of the function $f_n$ in terms of $\delta$. We do so by bounding the function $s_n$ from above by $\delta$. Using Lemma \ref{length}, we are able to obtain the estimate on $\gamma_0$, restricted to an interval $[a, b]$:
        $$
        \Len(\gamma_0|_{[a, b]})\Le \frac{\delta}{1 - \delta}
        $$

        \begin{figure}[hbt!]
            \label{est_simpl}
            \centering
            \includegraphics[height=2.5cm]{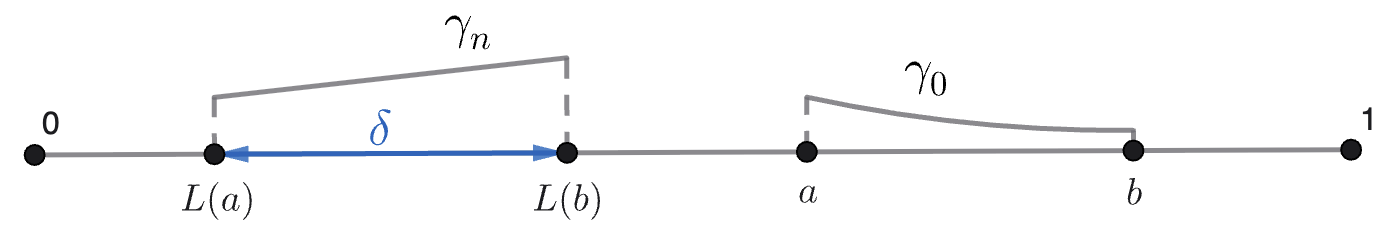}
            \caption{A sketch with notations for the proof of Theorem\ref{graphs}.}
        \end{figure}
        
        If $L$ was continuous, in fact uniformly continuous, we would be able to estimate $s_n$ from above by an arbitrarily small number. However, we cannot assume continuity of $L$, so we have to take an extra step.

        \item We partition the segment of $[L(a), L(b)]$ into $m$ equal intervals, and apply the previously obtained estimate to each of the smaller intervals. Since their length is $\frac{\delta}{m}$, the bound on the derivative of $\gamma_n$ for big $n$ is also significantly smaller. As a result, we obtain the following estimate:
        $$
        \Len(\gamma_0|_{[a, b]})\Le \frac{\delta}{1 - \frac{\delta}{m}}
        $$
        We apply this estimate for all positive integers $m$, and conclude that $\Len(\gamma_0|_{[a, b]})\Le \delta$.

        \item It follows from injectivity and other properties of $L$ that its image is a collection of intervals. Because of the estimate above the length of $\gamma_0$ is no greater than the sum of lengths of those intervals. Therefore the length of $\gamma_0$ is not greater than $1$. But the curve connecting $(0, 0)$ and $(1, 0)$ has length at least $1$ with equality only in the case of a straight segment. We conclude that $\gamma_0$ is in fact trivial.
    \end{itemize}

    \subsection{Supplementary definitions}\label{subsec5.1}
    Assume that a fixed curve $\gamma_0$ indeed remains a graph of a function under the dynamics of the map $\vphi\colon Y\to Y$. Denote by $f_n\colon [0, 1]\to\R$ the function, whose graph $\gamma_n$ is, so in the notations of Section \ref{sec4} we have: 
    
    $$
        \gamma_n(t) = \lt(x_n(t), f(x_n(t))\rt) = \lt(x_n(t), y_n(t)\rt).
    $$
    \\
    Let us define a single parameter $t$ on all the curves $\gamma_n$, in the following way:
    \begin{itemize}
        \item $\gamma_0(t) = (t, f_0(t))$. This is a well-defined parameter, since $\gamma_0$ is a graph of the function $f_0$;
    
        \item $\gamma_{n+1}(t)$ is defined inductively by the following formula:
        $$
        \gamma_{n+1}(t) = \gamma_n(t) + \frac{\dot{\gamma}_n(t)}{\|\dot{\gamma}_n(t)\|}-(1, 0).
        $$
    \end{itemize}
    This way the parameter $t$ "carries the memory" of the initial horizontal position of the point $\gamma_0(t)$.\\
    
    For the proof we will need to keep track of the horizontal movement of each point throughout the dynamics. Recall that for a curve $\gamma_n(t)$ we denote by $x_n(t)$ and $y_n(t)$ its horizontal and vertical coordinates respectively as functions of the parameter $t$. We introduce a sequence of functions $s_n\colon [0, 1]\to[0, 2]$, such that:
    \begin{align*}
        s_{n}(x) = 1 - \cos{\arctan{f_n'(x)}} = 1 - \frac{\dot{x}_n}{\sqrt{{\dot{x}_n^2 + \dot{y}_n^2}}},
    \end{align*}
    where the $'$ stands for $\frac{d}{dx}$, and $\dot{}$ stands for $\frac{d}{dt}$,
    
    \begin{figure}[hbt!]
        \centering
        \includegraphics[height=5cm]{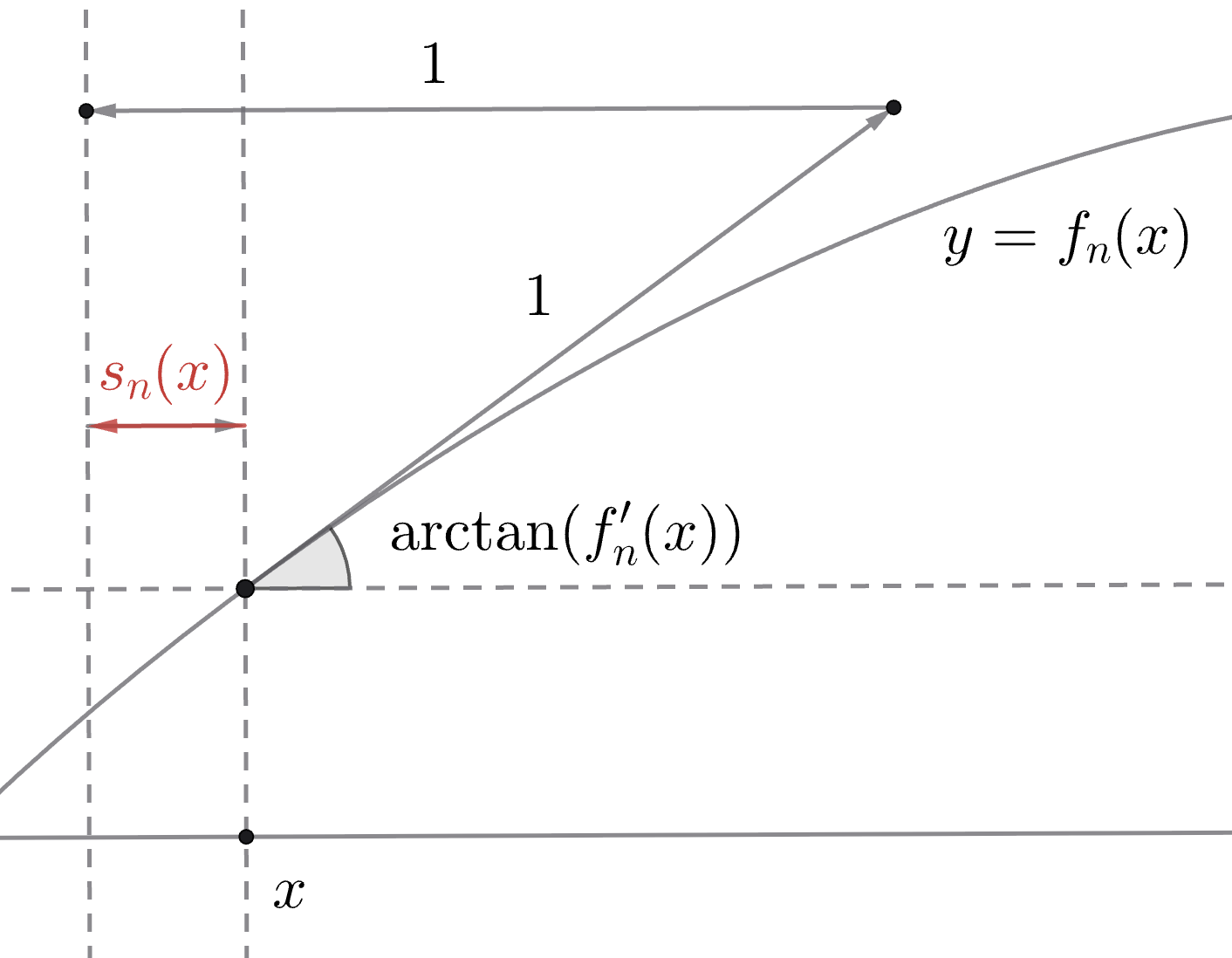}
        \caption{The definition of $s_n$.}
        \label{pics_nadv}
    \end{figure}
    
    \textit{Remark}: Since $f_n$ is infinitely differentiable and $\cos$ and $\arctan$ are smooth functions, we conclude that $s_n$ is infinitely differentiable as well.\\
    
    The intuition behind the functions $s_n$ comes from the following observation. 
    \begin{prop}
    \label{s_n}
        $x_{n+1}(t) = x_n(t) - s_n(x_n(t))$.
    \end{prop}
    
    \begin{proof}[Proof]
        As it is evident from the Figure \ref{pics_nadv}, $\cos\arctan{f'_n(x)}$ is the horizontal component of a unit tangent vector to $\gamma_n$ at the point $\lt(x(t), f_n(x(t))\rt)$ and it tracks the horizontal movement of a point with parameter $t$.\\
    
        Recall, that $\gamma_{n+1}(t) = \gamma_n(t) + \frac{\dot{\gamma_n}(t)}{\|\dot{\gamma_n}(t)\|} - (1, 0)$. We can equate the horizontal components of the left and the right sides of the latter formula to relate the functions $x_n(t)$ and $x_{n+1}(t)$:
        \begin{align*}
            x_{n+1}(t) = x_n(t) + \cos\arctan{f'_n(x_n(t))} - 1 = x_n(t) - s_n(x_n(t)),
        \end{align*}
        as wanted.
    \end{proof}

\subsection{Projection on horizontal direction.}

    \begin{prop}\label{x_nprop}
        The functions $x_n$ satisfy the following properties:
        \begin{enumerate}
            \item $x_n\in C^\Inf\lt([0, 1]\rt)$.\\
        
            \item $x_n(0) = 0,\,\, x_n(1) = 1$.\\
        
            \item For all $n\in\N$ and for all $0 \Le t_1 < t_2\Le 1$, the inequality $x_n(t_1)\Le x_n(t_2)$ holds.\\
        
            \item For all $n\in\N$ and for any fixed parameter value $t\in[0,1]$ we have $x_n(t)\Ge x_{n+1}(t)$.\\
        
        \end{enumerate}
    \end{prop}

    \begin{proof}
        Properties $1.$ through $3.$ follow immediately from the fact that $x_n$ is the horizontal projection of $\gamma_n$, which is a graph of $f_n$, and from Proposition \ref{s_n} and properties of the function $s_n$. Property $4.$ follows from Proposition \ref{s_n} as well, since $s_n$ is a positive function by definition.
    \end{proof}
    
    Also, these properties imply the following corollary:
    \begin{corollary}
    \label{limit}
        There exists a pointwise limit $\lim\lims_{n\to\Inf}x_n(t) = L\colon [0, 1]\to[0, 1]$.
        Moreover, we have $L(0) = 0,\, L(1) = 1$ and $L$ is non-decreasing.
    \end{corollary}
    \begin{proof}[Proof]
        For every fixed $t\in[0, 1]$ the sequence $x_n(t)$ is decreasing and bounded from below by $0$. Therefore for each $t\in[0, 1]$ there exists a limit $L(t) = \lim\lims_{n\to\Inf}x_n(t)$. It will be non-decreasing as a limit of continuous increasing functions.\\
        Since for $t = 0$ or $t=1$ the sequence $x_n(t)$ is constant and equal to $0$ and $1$ respectively, we conclude that $L(0) = 0$ and $L(1) = 1$.
    \end{proof}
    
    Monotonicity of $L$ suggests that it only can have type-one discontinuities, i.e., left and right limits of $L$ exist at all points, but they may not be equal to each other in a no more than countable set of points. In particular, the set $S = \im L$ is a no more than a countable collection of intervals.\\

\subsection{Estimating the length of segments of $\gamma_0$}

    The plan for the proof is to make an initial estimate on the length of certain segments of $\gamma_0$. Then we will refine those estimates, and conclude that $\gamma_0$ is no longer than $1$.\\
    
    Let us fix an arbitrary closed interval $[c, d]\subset S$, and denote by $\delta = d - c$ the length of this interval. Since $[c, d]\subset S$, we can define a non-empty interval $I_C^D = L^{-1}([c, d]) \subset [0, 1]$. It is indeed an interval, since $L$ is monotonous, $[c, d]\subset S = \im L$, and there are no discontinuities between $L^{-1}(c)$ and $L^{-1}(d)$.\\
    
    \textit{Remark}: The notation $I_C^D$ means that it is an interval between $C$ and $D > C$, but we do not specify whether the ends are open or closed. Apriori all four combinations of open and closed ends are possible in this case.\\
    
    Let's fix a closed interval $[a, b]\subset I_C^D$.\\
    
    Recall that $L = \lim\lims_{n\to\Inf}x_n$ is the limit of a decreasing sequence of functions, so for any $\veps > 0$ there exists $N\in\N$, such that for all $n > N$ the following holds:
    \begin{align}\label{coroll_ineq}
        c \Le L(a) < L(b) \Le x_n(b) < d + \veps.
    \end{align}

    \begin{figure}[hbt!]
            \label{config}
            \centering
            \includegraphics[height=2.5cm]{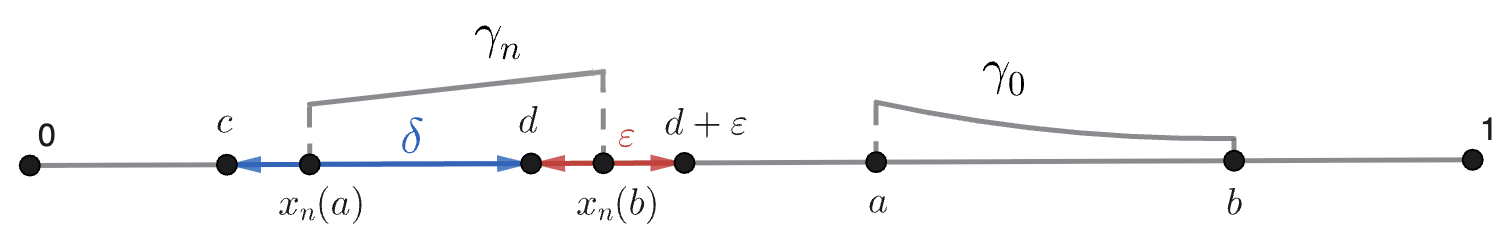}
            \caption{The configuration of points for large $n$.}
        \end{figure}

    There are two corollaries that follow from this simple observation:
    \begin{corollary}
    \label{cor1}
        In the above notations:
        $$
        x_n(b) - x_n(a) < \delta + \veps.
        $$
    \end{corollary}
    \begin{proof}
        Indeed, the monotonicity of the sequence $(x_n)$ and formula (\ref{coroll_ineq}) imply that:
        $$
        x_n(b) - x_n(a) < d + \veps - c = \delta + \veps.
        $$
    \end{proof}
    
    \begin{corollary}
    \label{cor2}
        In the above notations for all $x\in [x_n(a), x_n(b)]$ and for all $n > N$ the following holds:
        $$
        s_n(x) < d + \veps - c = \delta + \veps.
        $$
    \end{corollary}
    \begin{proof}[Proof]
        Recall, that the function $s_n(x) = 1 - \cos\arctan{f_n'(x)}$ tracks how much $\vphi$ moves each point to the left.\\
        The sequence $x_n(t)$ is a decreasing sequence with limit $L(t)$. From Proposition \ref{s_n} also follows that $L(t) = t - \sum\lims_{k=0}^\Inf s_k(x_k(t))$  with non-negative terms $s_n$. For $t\in[a, b]$ we have 
        $$
        L(a) \Le L(t) \Le x_n(t)\Le d + \veps, \quad\text{therefore}\quad s_n(x_n(t))\Le x_n(t) - L(t)\Le \delta + \veps.
        $$
    \end{proof}
    
    We will use Corollaries \ref{cor1} and \ref{cor2} to approximate the length of some segments of $\gamma_0$. To do that we will need the following obvious lemma:\\
    \begin{lemma}
    \label{length-est}
        Consider a smooth function $h\colon [p, q] \to \R$ with a bounded derivative $|h'(x)| < c$. Then the length of its graph $\Gamma_h$ satisfies the following inequality:
        $$
        Len(\Gamma_h) < |q - p|\cdot \lt(\cos{\arctan{c}}\rt)^{-1} = |q - p|\cdot\sqrt{c^2 + 1}.
        $$
    \end{lemma}
    
    Note, that the bound on the absolute value of the derivative $f_n'$ is equivalent to a certain bound on $-\cos\arctan f_n'$, since the function $\cos\arctan(s)$ has a global maximum at zero, and it is monotonously increasing on $(-\Inf, 0)$, and monotonously increasing on $(0, +\Inf)$. In particular, the equivalent form of the same statement is as follows:
    \begin{lemma}
    \label{impr-length-est}
        Consider a function $h\colon [p, q] \to \R$ such that $\cos\arctan h' > C$. Then the length of its graph $\Gamma_h$ satisfies the following inequality:
        $$
        Len(\Gamma_h) < \frac{|q-p|}{C}.
        $$
    \end{lemma}

    For any $[c, d]\subset S = \im L$ recall that $\delta = d - c$ and $I_C^D = L^{-1}([c, d])$.
    \begin{prop}
    \label{est1}
        For any closed interval $[a, b]\subset I_C^D$ the following inequality holds:
        $$
        Len\lt(\gamma_0|_{[a, b]}\rt) \Le \frac{\delta}{1 - \delta}.
        $$
    \end{prop}
    
    \begin{proof}[Proof]
        We can apply Lemma \ref{impr-length-est} to: 
        $$
        h = f_n\colon [x_n(a), x_n(b)]\to\R \,\,\text{with}\,\, \cos{\arctan{f_n'}} = 1 - s_n > 1 - \delta - \veps,
        $$
        and obtain that 
        $$
        Len(\gamma_n|_{[a, b]})\Le |\delta + \veps|\cdot \frac{1}{1 - \delta - \veps}
        $$
    
        According to Lemma \ref{length} this inequality yields $Len(\gamma_0|_{[a, b]}) \Le \frac{\delta + \veps}{1 - (\delta + \veps)}$. Since this estimate holds for all $\veps > 0$, we conclude that,
        $$
        Len(\gamma_0|_{[a, b]}) \Le \frac{\delta}{1 - \delta},
        $$
        as needed.
    \end{proof}
    
    However, this is not a sharp estimate. Here's an improved version of the previous proposition:
    
    \begin{prop}
    \label{est2}
        In the notation of Proposition \ref{est1} the following inequality holds:
        $$
        Len\lt(\gamma_0|_{[a, b]}\rt) \Le \delta
        $$
    \end{prop}
    
    \begin{proof}[Proof]
        Consider a partitioning of $[c, d]$ into $m$ equal closed intervals $K_j$ of length $\frac{\delta}{m}$, where $1 \Le j \Le m$. Define $I_j = L^{-1}(K_j)\cap[a, b]$. Note, that $\{I_j\}$ does not necessarily define a partition of $[a, b]$, since a preimage of a point could be a whole interval. However, $I_j$ do cover $[a, b]$ by definition.\\
    
        Apply Proposition \ref{est1} to each of the    segments $K_j$. We have the following estimates for all $1\Le j\Le m$:
        $$
        Len(\gamma_0|_{I_j})\Le \frac{\frac{\delta}{m}}{1 - \frac{\delta}{m}},
        $$
        $$
        Len(\gamma_0|_{[a, b]})\Le\sum_{j=1}^mLen(\gamma_0|_{I_j})\Le \sum_{j=1}^m\frac{\frac{\delta}{m}}{1 - \frac{\delta}{m}} = m\cdot \frac{\frac{\delta}{m}}{1 - \frac{\delta}{m}} = \frac{\delta}{1 - \frac{\delta}{m}}.
        $$
    
        Since this estimate holds for all natural $m$, we conclude that $Len(\gamma_0|_{[a, b]})\Le\delta$.
    \end{proof}
    
    \subsection{Final step}
        Fix a connected component of $S=\im L$, an interval $J$ of length $\delta_J$, and denote by $I_J = L^{-1}(J)$ its preimage under the map $L$.
        Since the preimage of every closed segment $[c, d]\subset J$ is a closed segment $[a, b]\subset I_J$, we can apply Proposition \ref{est2} to conclude that $Len(\gamma_0|_{[a, b]})\Le\delta_J$. This estimate holds for all closed segments $[a,b]\subset I_J$, so it is also true that $Len(\gamma_0|_{I_J})\Le\delta_J$.\\
        
        The following observation gives us the global estimate on the length of $\gamma_0$. Since $[0, 1] = \bigcup\lims_{J\subset S}I_J$, we conclude:
        $$
        Len(\gamma_0) \Le \sum_{J\subset S}Len(\gamma_0|_{I_J})\Le \sum_{J\subset S}\delta_J \Le 1.
        $$
        A curve that connects points $(0, 0)$ and $(1, 0)$ has length at least $1$, with the equality only for a straight segment. Therefore, $\gamma_0$ is indeed a straight segment, and $f_n\eq0$ for all positive integers $n$, as wanted.\quad$\blacksquare$

\section{Proof of Theorem\ref{hor-ampl}}
\label{proof2}
    \subsection{Supplementary definitions}
    Most of the definitions make sense without the assumption that the curves $\gamma_n$ are graphs of functions, so we will use the same notations as in the previous sections.
    
    Let's fix an arbitrary parameter $t\in[0, 1]$ on $\gamma_0$ and use the same parameter on all the curves $\gamma_n$, as we did it earlier:
    $$
    \gamma_{n+1}(t) = \gamma_n(t) + \frac{\dot{\gamma}_n(t)}{\|\dot{\gamma}_n(t)\|} - (1, 0).
    $$
    Recall that we define the functions $x_n(t)$ and $y_n(t)$ as cartesian coordinates of the curve $\gamma_n$. In other words, 
    $$
    \gamma_n(t) = (x_n(t), y_n(t)).
    $$
    Let us introduce a new notation --- the horizontal coordinate of the leftmost point of a curve $\gamma_n$:
    $$
    l_n = \min\lims_{t\in[0, 1]}x_n(t).
    $$
    Along with the horizontal coordinate of the rightmost point of $\gamma_n$, this notation will help us analyze the horizontal amplitude $H_n$ of the curves, which is defined in Section \ref{sec4}. However, we will not need a separate notation for the horizontal coordinate of the rightmost points of the curves, since the following proposition holds:
    \begin{prop}\label{rightbounds}
        Define $r_n = \max\lims_{t\in[0, 1]}x_n(t)$. Then for all positive integers $n$ the double inequality holds:
        \begin{align}
        \label{rights}
        1 \Le r_n\Le r_0.
        \end{align}
    \end{prop}
    \begin{proof}
        The functions $s_n$, defined in Subsection \ref{subsec5.1}, make sense in general, even without the assumption that all the curves $\gamma_n$ are graphs of functions. We take the parameter $t$ on the curve $\gamma_n$ as the argument of $s_n$. Since for all $n$ and for all $t$ we have $s_n(t) \Ge 0$, no point can move to the right as we apply $\vphi$, according to Proposition \ref{s_n}. Therefore, the sequence $r_n$ is non-increasing. At the same time, $\gamma_n(1) = (1, 0)$ for all $n$. Thus, the inequality (\ref{rights}) holds.
    \end{proof}
    By the definition of $H_n = H(\gamma_n)$ we conclude that $H_n = r_n - l_n \in \lt[|l_n - r_0|, |l_n - 1|\rt]$. Hence, it is enough to analyze the asymptotic growth of the sequence $\{l_n\}$, since $\{r_n\}$ is a bounded sequence.

    \subsection{Vertical Tangencies}
    It follows from Theorem \ref{graphs}, that there exists a natural number $N$, such that $\gamma_N$ is not a graph of a function $y = f(x)$. Since it is still a smooth curve connecting the points $(0, 0)$ and $(1, 0)$, we conclude that there exists a value of the parameter $t_v$ such that $\dot{\gamma}_N(t_v)$ is vertical.\\

    For all $n > N$ fix a parameter value $t_n\in[0, 1]$, such that $x_n(t_n) = l_n$. As the following proposition claims, all the leftmost points have vertical tangencies.

    \begin{prop}
        In above notations for all $n> N$ the vector $\dot{\gamma}_n(t_n)$ is vertical.
    \end{prop}
    \begin{proof}
        Here is the proof by induction starting from $N+1$.\\
        
        \textit{Base}: $\dot{\gamma}_{N+1}(t_{N+1})$ is vertical.\\
        Assume that $x_N(t_v) \Ge 1$. In particular, there might be parameters $t_r\in(0, 1)$ such that $x_N(t_r) \Ge 1$. In this case recall that $\dot{\gamma}_N(1) = (1, 0)$, which means that there exist parameters $t_l$ arbitrarily close to $1$ with $x_N(t_l) < 1$. Denote by $t_v'$ the parameter value at which $x_N(t)$ reaches minimum on a segment $[t_r, t_l]$. Since $x_N(t_l) < 1$, we conclude that $x_N(t_v') < 1$. Since $\gamma_N$ is smooth in both directions at $t_v'$, we conclude that the tangent vector to $\gamma_N(t_v')$ is vertical. Therefore, without loss of generality we can pick the value of $t_v$ such that $x_N(t_v) < 1$.\\

        \begin{figure}[hbt!]
            \label{est_vert}
            \centering
            \includegraphics[height=5cm]{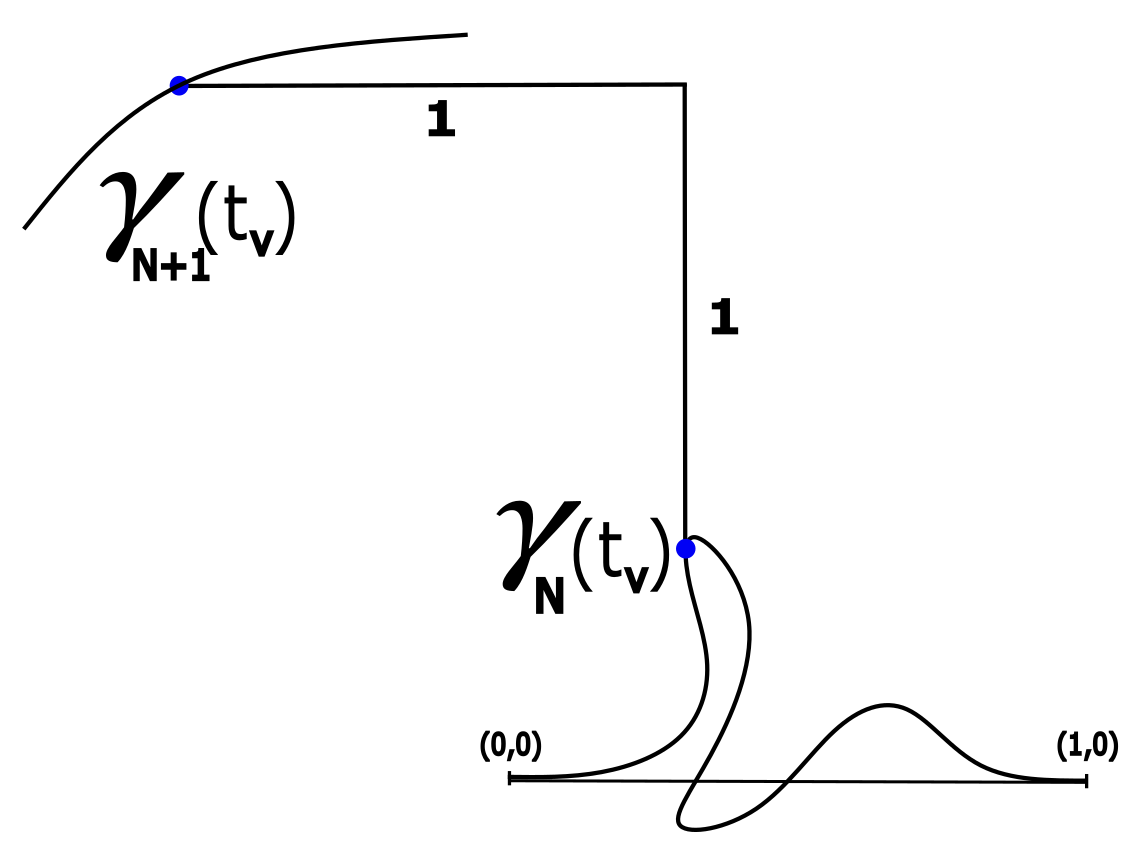}
            \caption{The image of a neighbourhood of a vertical tangency under $\vphi$.}
        \end{figure}
        
        Since $x_N(t_v) < 1$, we conclude that:
        $$
        x_{N+1}(t_v) = x_N(t_v) - 1 < 0\,\, \Rarr \,\,l_{N+1}=x_{N+1}(t_{N+1}) < 0.
        $$
        This inequality implies that $t_{N+1}\ne 0$, and therefore $\dot{\gamma}_{N+1}(t_{N+1})$ is vertical as a tangency direction of the leftmost point of a smooth curve.\\

        \textit{Step}: If $\dot{\gamma}_n(t_n)$ is vertical, then $\dot{\gamma}_{n+1}(t_{n+1})$ is vertical too.\\
        The proof is the exact copy of the proof of the Induction Base.
    \end{proof}
    
    \subsection{Estimates}
    The latter proposition has a corollary that gives us an estimate on the asymptotics of the sequence $\{l_n\}$:
    \begin{corollary}\label{lest1}
        In the above notations for all $n > N$ the inequality holds:
        $$
        l_{n + 1} \Le l_n - 1.
        $$
    \end{corollary}
    \begin{proof}
        The fact that $\gamma_n$ has a vertical tangency at $t_n$ implies that:
        $$
        x_{n+1}(t_n) = x_n(t_n) - 1 = l_n - 1 \,\,\Rarr\,\, l_{n+1} = x_{n+1}(t_{n+1}) \Le x_{n+1}(t_n) = l_n - 1,
        $$
        as needed.
    \end{proof} 

    On the other hand, the following observation follows from the general properties of the mapping $\vphi$:
    \begin{prop}\label{lest2}
        For all $n > N$ we have $l_n - l_{n+1} \Le 2$.
    \end{prop}
    \begin{proof}
        Indeed, by definition $s_n(t) < 2$ for all $n$ and for all $t$, so no point can travel further than $2$ to the left under the action of $\vphi$.
    \end{proof}
    Combine the results of the Corollary \ref{lest1} and Proposition \ref{lest2}, and conclude that:
    \begin{align}\label{lbounds}
        1\Le l_n - l_{n+1} \Le 2
    \end{align}
    The following chain of inequalities follows from the formula (\ref{lbounds}):
    $$
    l_N - k \Ge l_{N+k} \Ge l_N - 2k;
    $$
    $$
    l_N - (N+k - N) \Ge l_{N+k} \Ge l_N - (2N + 2k - 2N);
    $$
    $$
    l_N - (n - N) \Ge l_n \Ge l_N - (2n - 2N);
    $$
    $$
    r_n - l_N + (n - N) \Le r_n - l_n \Le r_n - l_N + (2n - 2N);
    $$
    $$
    n - (N + l_N - r_n) \Le H_n \Le 2n - (2N + l_N - r_n);
    $$
    $$
    n - (N + l_N - 1) \Le H_n \Le 2n - (2N + l_N - r_0).
    $$
    The transition to the last inequality follows from the bounds on $r_n$ we obtained in Proposition \ref{rightbounds}.\\

    Finally, we define $c_1 = N + l_N - 1$ and $c_2 = 2N + l_N - r_0$, and conclude:
    $$
    n - c_1 \Le H_n \Le 2n - c_2,
    $$
    as needed. \quad $\blacksquare$

\vskip 2cm

\section*{Acknowledgements}

I am grateful to my advisor Prof. Sergei Tabachnikov for his guidance, fruitful discussions, and for the useful corrections and remarks on the paper.

\end{document}